\documentclass[a4paper,10pt%,draft
]{amsart}
\usepackage{amsmath,amssymb,amsthm,}
\usepackage[alphabetic, abbrev]{amsrefs}
\usepackage{enumerate}
\usepackage{hyperref}
\usepackage[all]{xy}
\newdir{ >}{{}*!/-5pt/@{>}} %% cf xyguide exercise 14
\SelectTips{cm}{} %% cf xyguide page 9

\usepackage[pagewise,switch]{lineno}

\numberwithin{equation}{section}
\newtheorem{theorem}[subsection]{Theorem}
\newtheorem{corollary}[subsection]{Corollary}
\newtheorem{lemma}[subsection]{Lemma}
\newtheorem{proposition}[subsection]{Proposition}
\theoremstyle{definition}

\newtheorem{example}[subsection]{Example}

\newcommand{\cA}{\mathcal{A}}

\newcommand{\cC}{\mathcal{C}}
\newcommand{\cD}{\mathcal{D}}
\newcommand{\cE}{\mathcal{E}}
\newcommand{\cI}{\mathcal{I}}

\newcommand{\sset}{\mathrm{sSet}}
\newcommand{\ssetI}{\mathrm{sSet}^{\cI}}

\DeclareMathOperator{\hocolim}{hocolim}

\newcommand{\ot}{\leftarrow}

\newcommand{\iso}{\cong}
\newcommand{\tensor}{\otimes}
\newcommand{\bld}[1]{{\mathbf{#1}}}

\newcommand{\concat}{\sqcup}
\DeclareMathOperator{\colim}{colim}

\newcommand{\arxivlink}[1]{\href{http://arxiv.org/abs/#1}{\texttt{arXiv:#1}}}
\usepackage[pdftex]{graphics,color} 
\usepackage{eso-pic}
\usepackage{scrtime}
\usepackage{ifdraft}
\ifdraft{%
\AddToShipoutPicture{\put(30,30){\resizebox{!}{.4cm}%
   {{{\color[gray]{0.8}\texttt{%arxiv version 1, 
compiled \the\year-\the\month-\the\day~at \thistime}}}}}}
}{}

\ifdraft{%
  \newcommand{\Bemerkung}[1]{{\marginpar{\hspace{0.2\marginparwidth}\rule{0.6\marginparwidth}{0.75mm}\hspace{0.2\marginparwidth}}\noindent\bfseries[#1]}}
}{
  \newcommand{\Bemerkung}[1]{}
}

\title[Strictly commutative models for \texorpdfstring{$E_{\infty}$}{E-infinity} quasi-categories]{Strictly commutative models for\\ \texorpdfstring{$E_{\infty}$}{E-infinity} quasi-categories}

\author{Dimitar Kodjabachev}
\address{School of Mathematics and Statistics, 
University of Sheffield,
Hicks Building,
Sheffield, S3 7RH,
UK
}
\email{DKKodjabachev1@shef.ac.uk}

\author{Steffen Sagave} 
\address{Department of Mathematics and Informatics, 
Bergische Universit{\"a}t Wuppertal,
Gau{\ss}\-str. 20, 
42119 Wuppertal,
Germany} 
\email{sagave@math.uni-wuppertal.de}

\date{\today}

\begin{document}
\begin{abstract}
In this short note we show that $E_{\infty}$ quasi-categories can be replaced by strictly commutative objects in the larger category of diagrams of simplicial sets indexed by finite sets and injections. This complements earlier work on diagram spaces by Christian Schlichtkrull and the second author. 
\end{abstract}
\maketitle
\ifdraft{\linenumbers}{} % Creates line numbers with the lineno package.

%\tableofcontents
\section{Introduction}
An $E_{\infty}$ space is a space with a multiplicative structure
encoded by the action of an $E_{\infty}$ operad, i.e., an operad
consisting of contractible spaces with a free $\Sigma_n$-action. It is
shown in joint work by Christian Schlichtkrull and the second
author~\cite{Sagave-S_diagram} that $E_{\infty}$ spaces can be
rigidified to strictly commutative objects if one passes to a larger
category of $\cI$-spaces: if $\cI$ denotes the category of finite sets
$\bld{n}=\{1,\dots, n\}$ and injective maps, then the functor category
$\ssetI$ has a symmetric monoidal convolution product, and the
category $\ssetI[\cC]$ of commutative monoid objects in $\ssetI$
admits a model structure making it Quillen equivalent to the category
of $E_{\infty}$~spaces.

The following construction, due to Mirjam Solberg~\cite[Section 4.14]{Schlichtkrull-S_braided}, shows that symmetric monoidal categories give rise to commutative monoid objects in $\ssetI$ in a natural way.
\begin{example}\label{ex:nerve-of-sym-mon}
Let $(\cA,\tensor)$ be a symmetric monoidal category. We consider the functor $\Phi(\cA)\colon \cI \to \mathrm{Cat}$ with objects of $\Phi(\cA)(\bld{n})$ the $n$-tuples $(a_1,\dots, a_n)$ of objects in $\cA$ and morphisms
\[ \Phi(\cA)(\bld{n})((a_1,\dots, a_n),(b_1,\dots, b_n)) = \cA(a_1\tensor\dots\tensor a_n, b_1\tensor\dots\tensor b_n).
\] Functoriality in $\cI$ is induced by permutation of entries and insertion of the unit object of $\cA$. Composing with the nerve functor $N$ gives an $\cI$-simplicial set $N\Phi(\cA)$, and the symmetric monoidal structure of $\cA$ makes $N\Phi(\cA)$ a commutative monoid object in $\ssetI$, see~\cite[Proposition 4.16]{Schlichtkrull-S_braided}.
\end{example}

Equipped with the (standard or Kan) model structure, the category of
simplicial sets $\sset$ is Quillen equivalent to the category of
topological spaces. Therefore, weak homotopy types of spaces are
represented by simplicial sets. But simplicial sets also model
quasi-categories up to Joyal equivalence: there is a finer Joyal model
structure on $\sset$ whose fibrant objects are the quasi-categories
and whose weak equivalences are called Joyal equivalences (see
e.g.~\cite{Lurie_HTT} or~\cite{Dugger-S_mapping} for published
references). Simplicial sets with $E_{\infty}$ structures are also
interesting from this perspective since they model symmetric monoidal
$(\infty,1)$-categories~\cite{Lurie_HA}. These play a prominent role
in Lurie's work on the cobordism hypothesis~\cite{Lurie_cobordism}.

In view of these two interpretations of simplicial sets, it is an
obvious question if the above comparison of $E_{\infty}$ objects in
$\sset$ and strictly commutative objects in
$\ssetI$ still holds if we regard simplicial sets as
models for quasi-categories.  The aim of this note is to prove that
this is indeed the case:
\begin{theorem}\label{thm:lifted-model-intro}
\begin{enumerate}[(i)]
\item The category $\ssetI[\cC]$ of commutative monoid objects in
$\ssetI$ admits a left proper \emph{positive $\cI$-model structure} where
a map $f$ is a weak equivalence if and only if $\hocolim_{\cI}f$ is a Joyal 
equivalence. 
\item If $\cD$ is an $E_{\infty}$ operad, then there
is a chain of Quillen equivalences relating
$\ssetI[\cC]$ and the category $\sset[\cD]$ of
$E_{\infty}$ simplicial sets with the model structure lifted from the
Joyal model structure.
\end{enumerate}
\end{theorem}
Here an $E_{\infty}$ operad is an operad $\cD$ in simplicial sets such that $\cD(n)$ has a free $\Sigma_n$-action and $\cD(n)$ is contractible with respect to the Joyal model structure. Every $E_{\infty}$ operad in this sense is an $E_{\infty}$ operad in the classical sense since being contractible with respect to the Joyal model structure implies being contractible with respect to the Kan model structure. Moreover, every operad $\cD$ with $\cD(n)$ a $\Sigma_n$-free Kan complex such that $\cD(n)\to *$ is a weak homotopy equivalence is an $E_{\infty}$-operad in the sense of the theorem, for example the Barratt--Eccles operad whose $n$-th space is $E\Sigma_n$. 

By~\cite[Lemma 4.15]{Schlichtkrull-S_braided}, the object $N\Phi(\cA)$ in $\ssetI[\cC]$ considered in Example~\ref{ex:nerve-of-sym-mon} is fibrant in the model structure of Theorem~\ref{thm:lifted-model-intro}(i). It models the $E_{\infty}$ quasi-category $N\cA$ associated with the symmetric monoidal category $\cA$. Therefore Example~\ref{ex:nerve-of-sym-mon} shows that the nerve of a symmetric monoidal category can be rigidified to a commutative monoid object in $\ssetI$ in a natural way.

More generally, it follows from Theorem~\ref{thm:lifted-model-intro} that for any $E_{\infty}$ simplicial set $X$ in the sense of the theorem, there is an $A\in \ssetI[\cC]$ and a chain of maps ${A \ot B \to \mathrm{const}_{\cI}X}$ of $E_{\infty}$ objects in $\ssetI$ that induces a chain of Joyal equivalences when applying $\hocolim_{\cI}$ (compare~\cite[Corollary 3.7]{Sagave-S_diagram}). Hence the $E_{\infty}$ object $X$ can be replaced by the strictly commutative object $A$.  Although this rigidification of a structure up to homotopy by a strict one is in contrast to the philosophy of quasi-categories, we think that it is valuable to observe that $E_{\infty}$ quasi-categories can be expressed this way: when viewing simplicial sets as models for spaces, it is often easy to write down explicit objects in $\ssetI[\cC]$ that model $E_{\infty}$ spaces. This applies for example to $Q(X)$ if $X$ is connected~\cite[Example 1.3]{Sagave-S_diagram} or to $B\mathrm{GL}_{\infty}(R)^+$\cite[Remark 2.2]{Schlichtkrull_units}. It is likely that besides Example~\ref{ex:nerve-of-sym-mon} above,  there are more instances where interesting $E_{\infty}$ quasi-categories arise from commutative $\cI$-functors.

Theorem~\ref{thm:lifted-model-intro} and the corresponding statement about weak homotopy types of $E_{\infty}$ spaces~\cite[Theorem 1.2]{Sagave-S_diagram} refer to different model structures on the same categories that have the same cofibrations. Nonetheless, several arguments from~\cite{Sagave-S_diagram} do not apply here since~\cite[Theorem 1.2]{Sagave-S_diagram} was derived from a result about diagram spaces indexed by more general categories than $\cI$, and some of the more general arguments were based on special features of the Kan model structure. However, the Joyal model structure differs from the Kan model structure since it fails to be right proper and simplicial, and because it doesn't have an explicit set of generating acyclic cofibrations. In the proof of Theorem~\ref{thm:lifted-model-intro} presented here, we put emphasis on the points where new arguments are required and simply cite those parts of the proof of~\cite[Theorem 1.2]{Sagave-S_diagram} that also apply here. 

This note is a condensed and revised version of the first author's master's thesis at the University of Bonn, supervised by the second author. We thank an anonymous referee for a quick and helpful report on an earlier version of this note.
\section{Model structures on \texorpdfstring{$\cI$}{I}-simplicial sets}
The category of simplicial sets $\sset$ admits a \emph{Joyal model structure} with cofibrations the monomorphisms and fibrant objects the quasi-categories, i.e., the \emph{weak} or \emph{inner} Kan complexes. See~\cite[Theorem 2.2.5.1]{Lurie_HTT} or~\cite[Theorem 2.13]{Dugger-S_mapping}. The Joyal model structure is cofibrantly generated with generating cofibrations $I = \{\partial \Delta^n \to \Delta^n\,|\, n\geq 0\}$. We let $J$ be a set of generating acyclic cofibrations. (There is no known explicit description of such a set $J$.)

Let $\cI$ be the category with objects the finite sets $\bld{n}=\{1,\dots, n\}$ for $n\geq 0$ and morphisms the injections. Concatenation of ordered sets $\concat$ makes $\cI$ a symmetric monoidal category with unit $\bld{0}$ and symmetry isomorphism the obvious shuffle map.

Let $\ssetI$ be the functor category of $\cI$-diagrams of simplicial sets. For every object $\bld{n}$ of $\cI$, there is a free/forgetful adjunction $F^{\cI}_{\bld{n}}\colon \sset \rightleftarrows \ssetI\colon \mathrm{Ev}_{\bld{n}}$ with $F^{\cI}_{\bld{n}}(K) = \cI(\bld{n},-)\times K$ and $\mathrm{Ev}_{\bld{n}}(X) = X(\bld{n})$. For $X$ and $Y$ in $\ssetI$, the left Kan extension of the $\cI\times\cI$-diagram $X(-)\times Y(-)$ along $\concat \colon \cI\times \cI \to \cI$ defines an object $X\boxtimes Y$ in $\ssetI$. This construction defines a symmetric monoidal product $\boxtimes \colon \ssetI\times \ssetI\to \ssetI$ with unit $F^{\cI}_{\bld{0}}(*)$.

We now start to consider model structures on $\ssetI$.  Let $\cI_+$ be the full subcategory of $\cI$ on the objects $\bld{n}$ with $|\bld{n}|\geq 1$. We say that a map $f\colon X \to Y$ in $\ssetI$ is an absolute (resp. positive) level equivalence if $f\colon X(\bld{n})\to Y(\bld{n})$ is a Joyal equivalence for all $\bld{n}$ in $\cI$ (resp. all $\bld{n}$ in $\cI_+$), and an absolute (resp. positive) level fibration if $f\colon X(\bld{n})\to Y(\bld{n})$ is a fibration in the Joyal model structure for all $\bld{n}$ in $\cI$ (resp. all $\bld{n}$ in $\cI_+$). A map is an absolute (resp. positive) level cofibration if it has the left lifting property with respect to any map that is both an absolute (resp. positive) level fibration and level equivalence.
\begin{lemma}\label{lem:lev-model}
  These classes of maps define two cofibrantly generated left proper model structures on $\ssetI$, called the \emph{absolute} and the \emph{positive} level model structures.
\end{lemma}

\begin{proof}[Proof of Lemma~\ref{lem:lev-model}]
  The absolute case follows from~\cite[Theorem 11.6.1]{Hirschhorn_model}, and the
  positive case works as in~\cite[Proposition 6.7]{Sagave-S
    _diagram}. The sets
\begin{equation}\label{eq:gen-cof}
  I^{\mathrm{level}}_{\mathrm{abs}}= \{F^{\cI}_{\bld{n}}(i)\,|\, i\in
  I, \bld{n}\in \cI\}\text{ and }  I^{\mathrm{level}}_{\mathrm{pos}}=
  \{F^{\cI}_{\bld{n}}(i)\,|\, i\in I, \bld{n}\in \cI_+\}
\end{equation}
provide the generating cofibrations. The generating acyclic cofibrations $J^{\mathrm{level}}_{\mathrm{abs}}$ and $J^{\mathrm{level}}_{\mathrm{pos}}$ are defined similarly with $J$ in place of $I$.
\end{proof}
Since the Joyal model structure fails to be simplicial, the usual
Bousfield-Kan formula does not provide a homotopy invariant homotopy
colimit functor. In the following, $\hocolim_{\cI}\colon
\ssetI \to \sset$ denotes the functor constructed
in~\cite[\S 19]{Hirschhorn_model} using cosimplicial frames. We
recall from~\cite[Example 19.2.10]{Hirschhorn_model} that there is a
natural map $\hocolim_{\cI} X \to \colim_{\cI} X$.
\begin{lemma}\label{lem:hocolim-colim}
  If $X$ is absolute or positive level cofibrant in $\ssetI$, then the map $\hocolim_{\cI}
  X \to \colim_{\cI} X$ is a Joyal equivalence.
\end{lemma}
\begin{proof}
  This is analogous to~\cite[Theorem 19.9.1]{Hirschhorn_model}, with
  the absolute level model structure replacing the Reedy model
  structure in that reference.
\end{proof}
We say that a map $f\colon X \to Y$ in $\ssetI$ is an
$\cI$-equivalence if $\hocolim_{\cI}f$ is a Joyal equivalence of
simplicial sets, and an absolute (resp. positive) $\cI$-cofibration if it
is an absolute (resp. positive) level cofibration. A map is an absolute
(resp. positive) $\cI$-fibration if it has the right lifting property
with respect to any map that is both an absolute (resp. positive)
$\cI$-cofibration and an $\cI$-equivalence.
\begin{proposition}\label{lem:I-model}
  These classes of maps define two cofibrantly
  generated left proper model structures on $\ssetI$, called the \emph{absolute}
  and the \emph{positive} $\cI$-model structures.
\end{proposition}
We write $\ssetI_{\mathrm{abs}}$ and $\ssetI_{\mathrm{pos}}$ for these model categories. These (Joyal) $\cI$-model structures have the same cofibrations as the corresponding (Kan) $\cI$-model structures constructed in \cite[Proposition 6.16]{Sagave-S_diagram} by a different technique. 
\begin{proof}
  Since $\cI$ has an initial object, its classifying space is contractible. Hence the existence of the absolute $\cI$-model structure follows from~\cite[Theorem~5.2]{Dugger_replacing}, and we recall from~\cite{Dugger_replacing} that it is constructed as the left Bousfield localization of the absolute level model structure at $S = \{\alpha^* \colon F_{\bld{n}}^{\cI}(*) \to F_{\bld{m}}^{\cI}(*)\, | \, \alpha \colon \bld{m}\to\bld{n} \in \cI\}$.

  The positive $\cI$-model structure is defined to be the left Bousfield localization of the positive level model structure with respect to \[T = \{\alpha^* \colon F_{\bld{n}}^{\cI}(*) \to F_{\bld{m}}^{\cI}(*)\, | \, \alpha \colon \bld{m}\to\bld{n} \in \cI_+\}.\] It exists and is left proper by~\cite[Theorem 4.1.1]{Hirschhorn_model}. Hence it remains to show that its weak equivalences, the \emph{$T$-local equivalences}, are the $\cI$-equivalences. Since $T\subset S$, every $T$-local equivalence is an $\cI$-equivalence. Let $f$ be an $\cI$-equivalence. Passing to fibrant replacements, we may assume that $f$ is a map of $T$-local objects.  Restricting $f$ along the inclusion $\cI_+ \to \cI$ and applying~\cite[Theorem~5.2]{Dugger_replacing} to $\sset^{\cI_+}$, it follows that $\hocolim_{\cI_+}f$ is a Joyal equivalence. Since $\cI_+ \to \cI$ is homotopy cofinal \cite[Proof of Corollary 5.9]{Sagave-S_diagram}, this implies the claim.
\end{proof}
\begin{corollary}\label{cor:colim-Q-equiv} 
  There is a chain of Quillen equivalences \[\xymatrix@1{\ssetI_{\mathrm{pos}} \ar@<.15pc>[r]^{\mathrm{id}}& \ssetI_{\mathrm{abs}} \ar@<.15pc>[l]^{\mathrm{id}}\ar@<.15pc>[r]^-{\colim_{\cI}} & \sset \ar@<.15pc>[l]^-{\mathrm{const}_{\cI}}}\] relating $\ssetI$ equipped with the positive and absolute $\cI$-model structures and $\sset$ equipped with the Joyal model structure.
\end{corollary}
\begin{proof}
It is clear that $(\mathrm{id},\mathrm{id})$ is a Quillen equivalence. The adjunction $(\colim_{\cI}, \mathrm{const}_{\cI})$ is a Quillen equivalence by~\cite[Theorem~5.2(b)]{Dugger_replacing}.
\end{proof}
The next lemma and the subsequent proposition are analogous to~\cite[Proposition 7.1(iii)-(v) and Proposition 8.2]{Sagave-S_diagram}. The proofs given here avoid using features of the Bousfield-Kan formula for homotopy colimits. 
\begin{lemma}\label{lem:properties-level-cof}
\begin{enumerate}[(i)]
\item The gluing lemma for levelwise monomorphisms and level equivalences holds.
\item The gluing lemma for levelwise monomorphisms and $\cI$-equivalences holds.
\item For any ordinal $\lambda$ and any $\lambda$-sequence $(X_{\alpha})_{\alpha < \lambda}$ of levelwise monomorphisms, the canonical map $\hocolim_{\alpha < \lambda}X_{\alpha} \to \colim_{\alpha < \lambda}X_{\alpha}$ is a level equivalence. 
\end{enumerate}
\end{lemma}
\begin{proof}
Part (i) follows from the gluing lemma in left proper model categories~\cite[Proposition 13.5.4]{Hirschhorn_model}. Using (i) and the absolute level cofibrant replacement, it is enough to show (ii) for a diagram of absolute cofibrant objects. This special case follows from the gluing lemma in the Joyal model structure by applying $\colim_{\cI}$. Part (iii) follows from~\cite[Theorem 19.9.1]{Hirschhorn_model}. 
\end{proof}
\begin{proposition}\label{prop:cofibrant-boxtimes-preserves}
If $X$ is absolute cofibrant in $\ssetI$, then $X\boxtimes -$
preserves $\cI$-equivalences between not necessarily cofibrant objects. 
\end{proposition}
\begin{proof}
  We first assume that $X = F_{\bld{k}}^{\cI}(L)$ with $L \in
  \sset$ and $\bld{k} \in \cI$. Let $Y\to Z$ be an
  $\cI$-equivalence. If $Y$ and $Z$ are absolute cofibrant, then the
  claim follows by applying the strong symmetric monoidal functor
  $\colim_{\cI}$ and using Corollary~\ref{cor:colim-Q-equiv} and
  the pushout-product axiom for the Joyal model structure~\cite[2.15
  Proposition]{Dugger-S_mapping}. If $Y^c \to Y$ is an absolute level
  cofibrant replacement, then~\cite[Lemma 5.6]{Sagave-S_diagram} implies that $(X\boxtimes (Y^c \to Y))(\bld{m})$ is
  isomorphic to
  \begin{equation}\label{eq:cofibrant-boxtimes-preserves}
    L \times \colim_{\bld{k}\concat\bld{l}\to \bld{m}} Y^c(\bld{l}) \to L \times \colim_{\bld{k}\concat\bld{l}\to \bld{m}} Y(\bld{l}).
  \end{equation}
  Since each connected component of the comma category $\bld{k}\concat
  - \downarrow \bld{m}$ has a terminal object~\cite[Corollary
  5.9]{Sagave-S_diagram}, the colimits
  in~\eqref{eq:cofibrant-boxtimes-preserves} are Joyal equivalent to
  the corresponding homotopy colimits
  and~\eqref{eq:cofibrant-boxtimes-preserves} is a level
  equivalence. It follows that $X\boxtimes (Y \to Z)$ is an
  $\cI$-equivalence since $X\boxtimes (Y^c \to Z^c)$ is. With
  Lemma~\ref{lem:properties-level-cof} replacing those parts
  of~\cite[Proposition 7.1]{Sagave-S_diagram} that involve weak
  equivalences, the case of general $X$ follows as in the proof
  of~\cite[Proposition 8.2]{Sagave-S_diagram}.
\end{proof}
\begin{corollary}\label{cor:pproduct-monoid}
  The absolute and positive $\cI$-model structures on $\ssetI$ satisfy
  the pushout-product axiom and the monoid axiom.
\end{corollary}
\begin{proof}
  The part of the pushout-product axiom involving only cofibrations
  results from~\cite[Proposition 8.4]{Sagave-S_diagram}. As
  in~\cite[\S 8]{Sagave-S_diagram},
  Proposition~\ref{prop:cofibrant-boxtimes-preserves} implies the rest.
\end{proof}

The following lemma is analogous to~\cite[Lemma 8.1]{Sagave-S_diagram}.
\begin{lemma}\label{lem:quotient-by-free}
Let $G$ be a finite group and let $f\colon X \to Y$ and $Y \to E$ be  morphisms in $(\ssetI)^G$ such that  $\hocolim_{\cI}f$ is a Joyal equivalence. If $G$ acts freely on $E(\bld{m})$ for every object $\bld{m}$ in $\cI$, then $f/G \colon X/G \to Y/G$ is an $\cI$-equivalence. 
\end{lemma}
\begin{proof}
Since there is a $G$-map $Y\to E$, the $G$-action on $Y(\bld{m})$ is also free. Hence
$\hocolim_GY(\bld{m}) \to \colim_GY(\bld{m}) \iso (Y/G)(\bld{m})$ is a Joyal equivalence. Using the same argument for $X$, it follows that 
\[ \hocolim_G \hocolim_{\cI} f \simeq \hocolim_{\cI}\hocolim_G f \simeq \hocolim_{\cI}(f/G) \]
is a Joyal equivalence. 
\end{proof}

The use of the positive model structure is motivated by the positive
model structure for symmetric spectra discovered by Jeff Smith. %see~\cite{MMSS}.
The next lemma highlights one of its key features.
\begin{lemma}\label{lem:free-on-pos}
  If $X$ is positive $\cI$-cofibrant, then the $\Sigma_n$-action on
  the simplicial set $(X^{\boxtimes n})(\bld{m})$ is free
   for every object $\bld{m}$ of $\cI$.
\end{lemma}
\begin{proof}
  Let $f\colon U \to V$ and $U \to Y$ be maps in $\ssetI$. By a cell
  induction argument, it is enough to show that if $f$ is a generating
  cofibration and $\Sigma_n$ acts freely on $(Y^{\boxtimes
    n})(\bld{m})$ for every $\bld{m}$ in $\cI$, then $Z =
  Y\coprod_{U}V$ has this property. By~\cite[Lemma~A.8]{Sagave-S_diagram}, $Y^{\boxtimes n} \to Z^{\boxtimes n}$ has a
  filtration by maps that are cobase changes of maps of the form
  $\Sigma_{n}\times_{\Sigma_{n-i}\times\Sigma_i}Y^{\boxtimes
    n-i}\boxtimes f^{\Box i}$ where $f^{\Box i}$ is the $i$-fold
  iterated pushout product map in $(\ssetI,\boxtimes)$. Hence it
  suffices to show that $(Y^{\boxtimes n-i}\boxtimes f^{\Box
    i})(\bld{m})$ is a $(\Sigma_{n-i}\times\Sigma_i)$-projective
  cofibration of simplicial sets with
  $(\Sigma_{n-i}\times\Sigma_i)$-action. Since $f =
  F_{\bld{k}}^{\cI}(*)\times g$ with $g$ a generating cofibration for
  $\sset$ and $\bld{k}\in \cI_+$, it follows from \cite[Lemma
  5.6]{Sagave-S_diagram} that there is an isomorphism
  \begin{equation}\label{eq:free-on-pos}
    (Y^{\boxtimes n-i}\boxtimes f^{\Box i})(\bld{m}) \iso (\colim_{\bld{k}^{\concat i} \concat \bld{l} \to \bld{m}} Y^{\boxtimes n-i}(\bld{l})) \times g^{\Box i} 
  \end{equation}
  where $g^{\Box i}$ is the $i$-fold iterated pushout-product map of
  $g$ in $(\sset,\times)$. By~\cite[Corollary
  5.9]{Sagave-S_diagram}, each connected component of the indexing
  category $\bld{k}^{\concat i} \concat - \downarrow \bld{m}$ has a
  terminal object, and $\Sigma_i$ acts freely on the set of connected
  components. Hence $ \colim_{\bld{k}^{\concat i} \concat \bld{l} \to
    \bld{m}} Y^{\boxtimes n-i}(\bld{l})$ is a
  $(\Sigma_{n-i}\times\Sigma_i)$-free simplicial set,
  and~\eqref{eq:free-on-pos} is a $(\Sigma_{n-i}\times\Sigma_i)$-projective cofibration.
\end{proof}

\section{Model structures on structured diagrams of simplicial sets}
In the following, an operad $\cD$ denotes a sequence of simplicial
sets $\cD(n)$ with $\Sigma_n$-action such that $\cD(0)=*$, 
there is a unit map $*\to \cD(1)$, and there are structure maps
$\cD(n) \times \cD(i_1)\times \dots \times\cD(i_n) \to \cD(i_1+\dots+i_n)$ satisfying the usual associativity, unit and equivariance relations. %We say that $\cD$ 
It is called \emph{$\Sigma$-free} if $\Sigma_n$ acts
freely on $\cD(n)$ for all~$n$. 

Let $\ssetI[\cD]$ be the category of $\cD$-algebras in $(\ssetI,\boxtimes)$. We say that a model structure on $\ssetI$ lifts to $\ssetI[\cD]$ if $\ssetI[\cD]$ admits a model structure where a map is a weak equivalence or fibration if the underlying map in $\ssetI$ is. 

\begin{theorem}\label{thm:existence-operad-model-str}
Let $\cD$ be an operad. The positive $\cI$-model structure lifts to $\ssetI[\cD]$, and the absolute $\cI$-model structure lifts to $\ssetI[\cD]$ if $\cD$ is $\Sigma$-free. 
\end{theorem}
Since the generating cofibrations coincide, these model structures
have the same cofibrations as the corresponding Kan $\cI$-model
structures~\cite[Proposition 9.3]{Sagave-S_diagram}.
\begin{proof}
  As in the analogous statement about the Kan $\cI$-model structure~\cite[Proposition 9.3]{Sagave-S_diagram}, the claim reduces to showing that for a generating acyclic cofibration $f\colon U \to V$ in $\ssetI$, the bottom map in a pushout square
\[\xymatrix@-1pc{ 
\textstyle\coprod_{n\geq 0}\cD(n) \times_{\Sigma_n} U^{\boxtimes n} \ar[r] \ar[d] & \textstyle\coprod_{n\geq 0}\cD(n) \times_{\Sigma_n} V^{\boxtimes n}  \ar[d] \\ X \ar[r] & Y
}\]
in $\ssetI[\cD]$ is an $\cI$-equivalence. Replacing~\cite[Propositions 8.4 and 8.6]{Sagave-S_diagram} by Corollary~\ref{cor:pproduct-monoid} and~\cite[Lemma 8.1]{Sagave-S_diagram} by Lemma~\ref{lem:quotient-by-free}, the argument given in the proof of~\cite[Lemma 9.5]{Sagave-S_diagram} applies verbatim with one exception: we need to show that for any $n \geq 0$ and any $\bld{m}$ in $\cI$, the group $\Sigma_n$ acts freely on $V^{\boxtimes n}(\bld{m})$. Using that the generating cofibrations have cofibrant domains and codomains, we may assume that this also holds for the generating acyclic cofibrations~\cite[Corollaries 2.7 and 2.8]{barwick_left-right}. Hence the last claim follows from Lemma~\ref{lem:free-on-pos}.
\end{proof}
We recall that a morphism of operads $\Phi \colon \cD \to \cE$ induces an adjunction \[\Phi_*\colon \ssetI[\cD]\rightleftarrows \ssetI[\cE]\colon \Phi^*.\] 

\begin{proposition}\label{prop:change-of-operad}
Let $\Phi \colon \cD \to \cE$ be a morphism of operads with $\Phi_n\colon \cD(n)\to\cE(n)$ a Joyal equivalence for each $n\geq 0$. Then $(\Phi_*,\Phi^*)$ is a Quillen equivalence with respect to the positive $\cI$-model structures. If 
$\cD$ and $\cE$ are $\Sigma$-free, then it is also a Quillen equivalence with respect to the absolute $\cI$-model structures. 
\end{proposition}
\begin{proof}
Again the proof of the analogous statement about the Kan $\cI$-model structure~\cite[Proposition 9.12]{Sagave-S_diagram} applies almost verbatim: in the key
ingredient~\cite[Lemma 9.13]{Sagave-S_diagram}, Lemma~\ref{lem:properties-level-cof} replaces those parts of~\cite[Proposition 7.1]{Sagave-S_diagram} that involve weak equivalences, Corollary~\ref{cor:pproduct-monoid} replaces~\cite[Proposition 8.4]{Sagave-S_diagram}, Proposition~\ref{prop:cofibrant-boxtimes-preserves} replaces~\cite[Proposition 8.2]{Sagave-S_diagram}, and Lemma~\ref{lem:quotient-by-free} replaces~\cite[Lemma 8.1]{Sagave-S_diagram}.
\end{proof}

\begin{proof}[Proof of Theorem~\ref{thm:lifted-model-intro}]
  Part (i) follows from Theorem~\ref{thm:existence-operad-model-str}
  applied to the commutativity operad $\cC$ with $\cC(n)=*$ for every
  $n$. Left properness follows by the arguments from~\cite[Lemma 11.8
  and Proposition 11.9]{Sagave-S_diagram}, where again the results
  from Section~2 replace the corresponding statements
  in~\cite{Sagave-S_diagram}.

  If $\cD$ is an $E_{\infty}$ operad, then there is a canonical
  morphism $\Phi\colon \cD \to \cC$, and we obtain a chain of Quillen
  adjunctions
\[
\xymatrix{\ssetI_{\mathrm{pos}}[\cC] \ar@<-.15pc>[r]_{\Phi^*}& \ssetI_{\mathrm{pos}}[\cD] \ar@<.15pc>[r]^{\mathrm{id}}\ar@<-.15pc>[l]_{\Phi_*} & \ssetI_{\mathrm{abs}}[\cD] \ar@<.15pc>[l]^{\mathrm{id}}\ar@<.15pc>[r]^-{\colim_{\cI}} &  \sset[\cD] \ar@<.15pc>[l]^-{\mathrm{const}_{\cI}}}
\]
The first adjunction is a Quillen equivalence by Proposition~\ref{prop:change-of-operad}. The last two adjunctions are Quillen equivalences by Corollary~\ref{cor:colim-Q-equiv} and the fact that cofibrant objects in $ \ssetI_{\mathrm{abs}}[\cD]$ are cofibrant in $\ssetI_{\mathrm{abs}}$ if $\cD$ is $\Sigma$-free~\cite[Corollary 12.3]{Sagave-S_diagram}. %This shows (ii).
\end{proof}
%\enlargethispage{3ex}

%\bibliography{Joyal-I}

\begin{bibdiv}
\begin{biblist}

\bib{barwick_left-right}{article}{
      author={Barwick, Clark},
       title={On left and right model categories and left and right {B}ousfield
  localizations},
        date={2010},
        ISSN={1532-0073},
     journal={Homology, Homotopy Appl.},
      volume={12},
      number={2},
       pages={245\ndash 320},
         url={http://projecteuclid.org/euclid.hha/1296223884},
}

\bib{Dugger-S_mapping}{article}{
      author={Dugger, Daniel},
      author={Spivak, David~I.},
       title={Mapping spaces in quasi-categories},
        date={2011},
        ISSN={1472-2747},
     journal={Algebr. Geom. Topol.},
      volume={11},
      number={1},
       pages={263\ndash 325},
         url={http://dx.doi.org/10.2140/agt.2011.11.263},
}

\bib{Dugger_replacing}{article}{
      author={Dugger, Daniel},
       title={Replacing model categories with simplicial ones},
        date={2001},
        ISSN={0002-9947},
     journal={Trans. Amer. Math. Soc.},
      volume={353},
      number={12},
       pages={5003\ndash 5027 (electronic)},
         url={http://dx.doi.org/10.1090/S0002-9947-01-02661-7},
}

\bib{Hirschhorn_model}{book}{
      author={Hirschhorn, Philip~S.},
       title={Model categories and their localizations},
      series={Mathematical Surveys and Monographs},
   publisher={American Mathematical Society},
     address={Providence, RI},
        date={2003},
      volume={99},
        ISBN={0-8218-3279-4},
}

\bib{Lurie_HTT}{book}{
      author={Lurie, Jacob},
       title={Higher topos theory},
      series={Annals of Mathematics Studies},
   publisher={Princeton University Press, Princeton, NJ},
        date={2009},
      volume={170},
        ISBN={978-0-691-14049-0; 0-691-14049-9},
}

\bib{Lurie_cobordism}{incollection}{
      author={Lurie, Jacob},
       title={On the classification of topological field theories},
        date={2009},
   booktitle={Current developments in mathematics, 2008},
   publisher={Int. Press, Somerville, MA},
       pages={129\ndash 280},
}

\bib{Lurie_HA}{misc}{
      author={Lurie, Jacob},
       title={Higher algebra},
        note={Preprint, available at
  \url{http://www.math.harvard.edu/~lurie/}},
}

\bib{Schlichtkrull_units}{article}{
      author={Schlichtkrull, Christian},
       title={Units of ring spectra and their traces in algebraic
  {$K$}-theory},
        date={2004},
        ISSN={1465-3060},
     journal={Geom. Topol.},
      volume={8},
       pages={645\ndash 673 (electronic)},
}

\bib{Sagave-S_diagram}{article}{
      author={Sagave, Steffen},
      author={Schlichtkrull, Christian},
       title={Diagram spaces and symmetric spectra},
        date={2012},
        ISSN={0001-8708},
     journal={Adv. Math.},
      volume={231},
      number={3-4},
       pages={2116\ndash 2193},
         url={http://dx.doi.org/10.1016/j.aim.2012.07.013},
}

\bib{Schlichtkrull-S_braided}{misc}{
      author={Schlichtkrull, Christian},
      author={Solberg, Mirjam},
       title={Braided injections and double loop spaces},
        note={\arxivlink{1403.1101}},
}

\end{biblist}
\end{bibdiv}

\end{document}